\documentclass[11pt,a4paper]{amsart}
\usepackage{amsmath,amsfonts,amsthm,amsopn,color,amssymb,enumitem}
\usepackage{palatino}
\usepackage{graphicx}
\usepackage[english]{babel}
\usepackage{caption}
\usepackage{subcaption}
\usepackage[colorlinks=true]{hyperref}
\hypersetup{urlcolor=blue, citecolor=red, linkcolor=blue}
\usepackage{cite}
\usepackage{esint}
\usepackage[title]{appendix}

\newtheorem{theorem}{Theorem}[section]
\newtheorem{lemma}[theorem]{Lemma}

\newtheorem{proposition}[theorem]{Proposition}
\newtheorem{remark}[theorem]{Remark}

\newtheorem{thmx}{Theorem}

\newcommand{\e}{\varepsilon}
\newcommand{\R}{\mathbb{R}}

\newcommand{\weakto}{\rightharpoonup}

\newcommand{\N}{\mathbb{N}}

\def\bbm[#1]{\mbox{\boldmath $#1$}}
\newcommand{\beq }{\begin{equation}}
\newcommand{\eeq }{\end{equation}}

\def\sideremark#1{\ifvmode\leavevmode\fi\vadjust{\vbox to0pt{\vss% the remark3
 \hbox to 0pt{\hskip\hsize\hskip1em%                          will appear only
 \vbox{\hsize3cm\tiny\raggedright\pretolerance10000%          on the side
  \noindent #1\hfill}\hss}\vbox to8pt{\vfil}\vss}}}%

\begin{document}

\title[Compactly supported steady solutions of the 2D Euler equations]{Symmetry results for compactly supported steady solutions of the 2D Euler equations}

\author{David Ruiz}
  \address{David Ruiz \\
    IMAG, Universidad de Granada\\
    Departamento de An\'alisis Matem\'atico\\
    Campus Fuentenueva\\
    18071 Granada, Spain}
  \email{daruiz@ugr.es}

\thanks{D. R. has been supported by the FEDER-MINECO Grant PGC2018-096422-B-I00 and by J. Andalucia (FQM-116). He also acknowledges financial support from the Spanish Ministry of Science and Innovation (MICINN), through the \emph{IMAG-Maria de Maeztu} Excellence  Grant CEX2020-001105-M/AEI/10.13039/501100011033.}

% Type down your paper title

\keywords{Euler equations; circular flows; semilinear elliptic equations; free boundary problems; method of moving planes.}

\subjclass[2010]{35B06; 35B50; 35B53; 35J61; 76B03.}

% The abstract
\begin{abstract}
In this talk we present some results regarding compactly supported solutions of the 2D steady Euler equations. Assuming that $\Omega = \{x \in \R^2:\ u(x) \neq 0\}$ is an  annular domain, we prove that the streamlines of the flow are circular. We are also able to remove the topological condition on $\Omega$ if we impose regularity and nondegeneracy assumptions on $u$ at $\partial \Omega$. The proof uses that the corresponding stream function solves an elliptic semilinear problem $-\Delta \phi = f(\phi)$ with $\nabla \phi=0$ at the boundary. One of the main difficulties in our study is that $f$ is not Lipschitz continuous near the boundary values. However, $f(\phi)$ vanishes at the boundary values and then we can apply a local symmetry result of F. Brock to conclude. 

In the case $\partial_{\nu} u \neq 0$ at $\partial \Omega$ this argument is not possible. In this case we are able to use the moving plane scheme to show symmetry, despite the possible lack of regularity of $f$. We think that such result is interesting in its own right and will be stated and proved also for higher dimensions. The proof requires the study of maximum principles, Hopf lemma and Serrin corner lemma for elliptic linear operators with singular coefficients.

\end{abstract}

\maketitle

\section{Introduction}
\setcounter{equation}{0}

In this paper we study stationary solutions of the 2D Euler equations:

\begin{equation} \label{1} \left \{ \begin{array}{ll}  u \cdot \nabla u & = - \nabla p, \\ div \, u & =0.\end{array}  \right. \end{equation}

In particular we are interested in nonzero compactly supported solutions of the above problem. In dimension 3, the existence or not of such solutions has been an open problem for many years. In \cite{jiu} it is proved that if $u$ is a compactly supported axisymmetric solution of \eqref{1} with no swirl, then $u=0$. Another rigidity result was given in \cite{nad, chae} for Beltrami fields of finite energy. But a general answer to this question was open until the recent work \cite{gavrilov}, where a nonzero compactly supported solution is built. This example has been further improved in \cite{constantin} where it is also extended to other fluid equations. A different construction of a solution, piecewise smooth but discontinuous, has been recently given in \cite{enciso} 

The existence of compactly supported solutions for \eqref{1} in dimension 2 is a much simpler question. If $\omega$ is the vorticity of the fluid, then:

\begin{equation} \label{0parallel} \nabla \omega \cdot u =0. \end{equation}

Let us denote by $\phi: \R \to \R $ the corresponding stream function (that is, $u= \nabla^{\perp} \phi$ and then $\Delta \phi = \omega$). In this way, \eqref{0parallel} reduces to:
\begin{equation} \label{preparallel} \nabla (\Delta \phi) \cdot \nabla^{\perp} \phi =0. \end{equation}

In words, the gradient of the stream function and the gradient of its laplacian must be parallel. It is clear that \eqref{preparallel} is satisfied for any (say, $C^2$) radially symmetric function $\phi$ with compact support. Of course this is also true if $\phi = \phi_1 + \phi_2$, where $\phi_i$ are radially symmetric functions with respect to points $p_i\in \R^2$ and with disjoint support. 

Observe, however, that for all such examples the streamlines of the flow are circular lines (with possibly different centers). One could wonder whether there exists a compactly supported solution to \eqref{1} in dimension 2 with noncircular streamlines. This question is the main motivation of our work. 

We could not find in the literature any condition on compactly supported solutions of the 2D Euler equations that leads to radial symmetry. There are some related results available, though. In \cite{hamel4} radial symmetry is proved for solutions in bounded domains under constant tangential velocity at the boundary: however this constant is not allowed to be $0$. Another symmetry result is \cite{gs1} for nonnegative and compactly supported vorticity, and here the velocity field need not have compact support (it is an immediate consequence of the divergence theorem that the unique compactly supported velocity field with nonnegative vorticity is $0$). 

On the other hand, there are very recent constructions of compactly supported solutions with noncircular streamlines. In \cite{gs2} nontrivial patch solutions with three layers are built; here $\omega$ has the form $\sum_{i=1}^3 c_i 1_{D_i}$ for some $c_i \in \R$ and some domains $D_i$ which are perturbations of concentric disks. In the forthcoming paper \cite{noi} a different example is given via a nonradial solution of a semilinear problem in a perturbed annulus. In both cases, the solutions are found by using a local bifurcation argument, and in both cases the velocity fields fail to be $C^1$ in $\R^2$.   

%\begin{equation}  \left \{ \begin{array}{rll}  -\Delta \phi  &= f(\phi), & \mbox{ in } D, \\ \nabla \phi  &=0,  & \mbox{ in } \partial D, \end{array} \right. \end{equation}	
%where $D$ is a perturbation of an annulus. 

Our first result is the following:

\begin{thmx}\label{teo} Let $u: \R^2 \to \R^2$ be a compactly supported $C^1$ solution of \eqref{1}, and define $\Omega = \{x \in \R^2:\ u(x) \neq 0\}$. We assume that:	
	\begin{enumerate}
		
		\item[(A1)] $\Omega$ is a $C^0$ annular domain, i.e., $\Omega= G_0 \setminus \overline{G_1}$, where $G_i$ are $C^0$ simply connected domains and $\overline{G_1} \subset G_0$. 
		\item[(A2)] $u$ is of class $C^2$ in $\Omega$.
		
		\end{enumerate} 
	%Here $\nu$ stands for the outer normal unit vector in $\partial \Omega$. 
	%\
	
\medskip Then $\Omega$ is an annulus and $u(x)$ is a circular vector field. Being more specific, there exist $p \in \R^2$ and $0<R_1<R_2$ with $\Omega= A(p; R_1, R_2)$, and a certain function $V$ such that
	$$ u(x)= V(|x-p|) (x-p)^{\perp}.$$
\end{thmx}

Let us emphasize that the regularity condition in assumption (A1) prevents the appearance of isolated stagnation points. This is a rather typical assumption in many rigidity results in this fashion, see \cite{hamel1, hamel2, hamel3, hamel4}. Moreover, the set $\Omega$ is assumed to be of annular type. If we impose some regularity and nondegeneracy on $\partial \Omega$, this topological requirement can be dropped:

\begin{thmx} \label{teo2} Let $u: \R^2 \to \R^2$ be a compactly supported solution of \eqref{1}, and define $\Omega = \{x \in \R^2:\ u(x) \neq 0\}$. Assume that for some integer $k \geq 2$,

	\begin{enumerate}
		
		\item[(B1)] $\Omega$ is a $C^{1}$ domain, $u \in C^2(\Omega)$ and $u$ is of class $C^k$ in a neighborhood of $\partial \Omega$ relative to $\overline{\Omega}$.
		
		\item[(B2)] $\frac{\partial^j u}{\partial \nu^j} = 0$  if $ j < k$, and $\frac{\partial^k u}{\partial \nu^k}  \neq 0$ for all $x \in \partial \Omega.$
		
	\end{enumerate} 
	%Here $\nu$ stands for the outer normal unit vector in $\partial \Omega$. 
	%\
Then the assertion of Theorem \ref{teo} holds true.
\end{thmx}

The first step in the proof of Theorem \ref{teo2} is to show that $\Omega$ is an annular domain. This is based on a topological observation that makes use of the Brouwer degree and the lack of stagnation points of $u$ in $\Omega$. We would like to point out that this observation works also in the framework of \cite[Theorem 1.13]{hamel4}. As a consequence, the result there holds without assuming a priori that $\Omega$ is an annular domain. See Proposition \ref{annular0} and Remark \ref{delicate} for more details.

Then we are led to Theorem \ref{teo}. In its proof we are largely inspired by the works of Hamel and Nadirashvili, \cite{hamel1, hamel2, hamel3, hamel4}. In particular, Theorem \ref{teo} can be seen as a version of \cite[Theorem 1.13]{hamel4} with a degenerate boundary condition. The main idea is to show that the stream function $\phi$ solves a semilinear equation, and then to use symmetry results for this kind of problems to conclude.

Let us explain the argument in more detail. First we will show that the level sets of $\phi$ in $\Omega$ are connected  curves; this is a consequence of the lack of stagnation points in $\Omega$. As a consequence we can show that, up to addition and multiplication by constants, the stream function $\phi$ solves a semilinear elliptic problem:
	\begin{equation} \label{preod} \left \{ \begin{array}{rll}  -\Delta \phi  &= f(\phi), \ \ \ \phi  \in (0,1)  &\mbox{ in } \Omega, \\ \phi  &=0,  \qquad \ \nabla \phi=0, & \mbox{ in } \partial G_0, \\ \phi &= 1, \qquad \nabla \phi=0,  &\mbox{ in }\partial G_1, \end{array} \right. \end{equation}	
for some continuous function $f:[0,1]\to \R$. 

We would like to mention that in \cite[Theorem 1.13]{hamel4} one has nonzero constant Neuman derivative at the boundary. Then, the function $f$ is $C^1$ and one can use the results of \cite{reichel, sirakov} to conclude that $\phi$ is radially symmetric and $\Omega$ is an annulus. These results are based on the well-known technique of moving planes, applied to overdetermined elliptic problems as in \cite{serrin}.

However, in our situation $f$ is $C^1$ in $(0,1)$, but $f$ is not differentiable at $0$, $1$, nor even Lipschitz continuous. This is a consequence of the vanishing of $u$ at $\partial \Omega$. Moreover, we do not have any information on the monotonicity of $f$ near $0$ and $1$.

A related result is given in \cite[Theorem 1.10]{hamel4} for simply connected domains with one stagnation point. There $f$ can fail to be Lipschitz at the maximum value, and the authors are able to apply the moving plane method to obtain symmetry. However, the argument there is linked to the fact that the stagnation point is unique, and cannot be translated to our framework.

We emphasize that this is not only a technical question; it is known that in general the moving plane method fails for non Lipschitz continuous functions, see for instance \cite{gnn, brock}. Still, there are some symmetry results without monotonicity or Lipschitz regularity assumptions, that we briefly review below.

A strategy using Pohozaev identities and isoperimetric inequalities was introduced by Lions in \cite{lions} (see also \cite{kesavan-pacella, serra}) but it works only when $\Omega$ is a ball and $f(\phi) \geq 0$. Observe that in our case,
$$ \int_{\Omega} f(\phi) =0,$$
and hence $f$ changes sign. Another strategies to deduce symmetry for non-Lipschitz nonlinearities are the continuous Steiner symmetrization (developed by F. Brock in \cite{brock0, brock}) and a careful use of the moving plane method under some conditions on $f$ (see \cite{dolbeault}). In short, they imply that if $\Omega$ is a ball, then $\phi$ is radially symmetric in a countable number or balls or annuli, and we can have plateaus outside. 

By the $C^1$ regularity of $u$ we have that $Du=0$ on $\partial \Omega$: from this one obtains that $f(0)=f(1)=0$. Hence the function $\phi$ solves the semilinear equation $-\Delta \phi = f(\phi)$ in the whole plane $\R^2$. This allows us to use the general local symmetry results of \cite{brock0, brock}. Since $\nabla \phi$ does not vanish in $\Omega$, the existence of plateaus inside $\Omega$ is excluded and we finish the proof.

It is worth pointing out that the proof works also if $\Omega$ is a punctured simply connected domain (as in \cite[Theorem 1.10]{hamel4}), with almost no changes . This can be seen as a degenerate case of Theorem \ref{teo}, where $G_1$ is collapsed to a single point. See Theorem \ref{teo4} in Section 6 for details.

\medskip 

We are able to give a symmetry result also in the case  $\partial_{\nu} u \neq 0$ in $\Omega$, which would correspond to the case $k=1$ in Theorem \ref{teo2}. In such case $u$ fails to be $C^1$ in $\R^2$, so the formulation of the result needs to be slightly changed.

\begin{thmx} \label{teo3} Let $\Omega \subset \R^2$ a $C^2$ bounded domain and $u \in C^2(\overline{\Omega}, \R^2)$ a solution of the problem:
\begin{equation} \label{eq1} \left \{ \begin{array}{rll}  u \cdot \nabla u & = - \nabla p & \mbox{ in } \Omega, \\ div \, u & =0  & \mbox{ in } \Omega, \\ u & =0 & \mbox{ in } \partial \Omega. \end{array} \right. \end{equation}
Assume that: 

\begin{enumerate}

\item[(C1)] $u(x) \neq 0$ for all $x \in \Omega$.

\item[(C2)] $\partial_{\nu} u(x) \neq 0$ for all  $x \in \partial \Omega.$
\end{enumerate} 
%Here $\nu$ stands for the outer normal unit vector in $\partial \Omega$. 
%\

Then the assertion of Theorem \ref{teo} holds true.
\end{thmx}

%\begin{theorem}[\cite{hamel4}]
%	
%	 \label{teo2} Let $\Omega \subset \R^2$ be a $C^2$ annular domain, that is, 
%\begin{equation} \label{annular} \Omega = G_0 \setminus G_1, \end{equation}
%where $G_i$ are bounded simply connected domains and $\overline{G_1} \subset G_0$.
% Let $u \in C^2(\overline{\Omega}, \R^2)$ be a solution of the boundary value problem:
%	
%	\begin{equation} \label{eq2} \left \{ \begin{array}{rll}  u \cdot \nabla u & = - \nabla p & \mbox{ in } \Omega, \\ div \, u & =0 , & \mbox{ in } \Omega, \\ u \cdot \nu & =0 & \mbox{ in }\partial G_0 \cup \partial G_1, \\ |u| &= c_i & \mbox{ in } \partial G_i, \ i=0, 1,\end{array} \right. \end{equation}
%for some constants $c_i>0$. Assume that: 
%	
%	\begin{enumerate}
%		\item[(A1)] $u(x) \neq 0$ for all $x \in \Omega$.
%	\end{enumerate} 
%	
%	Then $\Omega$ is an annulus and $u(x)$ is a circular vector field, that is, there exists $p \in \R^2$ and $0<R_1<R_2$ with $\Omega= A(p; R_1, R_2)$ and 
%	
%	$$ u(x)= V(|x-p|) (x-p)^{\perp},$$
%	
%	for certain function $V$.
%	
%	
%\end{theorem}
%
The proof of Theorem \ref{teo3} uses the same ideas presented before to conclude that $\Omega$ is an annular domain and  $\phi$ solves \eqref{preod}, but now $\phi$ is defined only in $\overline{\Omega}$. As commented above, $f$ is $C^1$ in $(0,1)$ but we cannot assure that $f$ is Lipschitz nor monotone near $0$ and $1$. Moreover, by assumption (C2), $f(0) \neq 0$, $f(1) \neq 0$. Hence, if we extend $\phi$ as constants outside $\Omega$, we do not have a solution of a semilinear equation anymore, and hence the result of \cite{brock} is not applicable.

In this case  we are able to perform the moving plane technique to conclude symmetry. The main reason is that we have a good knowledge of the behavior of $\phi$ near the boundary precisely by assumption (C2). 

Let us give a short explanation of the idea. Generally speaking, the moving plane method is based on comparing the function $\phi$ with $\psi= \phi \circ \pi$, where $\pi$ is the reflection with respect to an hyperplane intersecting the domain. It turns out that

$$ -\Delta (\psi - \phi)+ c(x) (\psi - \phi)=0,$$
where 
$$ c(x)= \left \{\begin{array}{ll} \displaystyle - \frac{f(\psi(x)) - f(\phi(x))}{\psi(x) - \phi(x)} & \mbox{ if } \phi(x) \neq \psi(x), \\ 0 &  \mbox{ if } \phi(x) = \psi(x). \end{array} \right. $$

\smallskip If $f$ is Lipschitz continuous, then $c(x) \in L^\infty(\Omega)$ and we can apply the well-known properties of the operator $- \Delta + c(x)$. In our case, however, $c(x)$ can be singular at $\partial \Omega$. The key observation here is that thanks to (C2) we can control the singularity of $c(x)$ at $\partial \Omega$; indeed,
$$ c(x) d(x) \in L^{\infty}(\Omega),$$
where $d(x)= dist(x, \partial \Omega)$. In this paper we are able to prove the desired properties for the operator $- \Delta + c(x)$ for such coefficient $c(x)$. Being more specific, we will prove weak and strong maximum principles, the Hopf lemma and the Serrin corner lemma for this kind of operators. With those ingredients in hand, the moving plane strategy can be applied to prove that $\Omega$ is an annulus and $\phi$ is radially symmetric. We think that this result is of independent interest, and it is stated and proved for any dimension.

The rest of the paper is organized as follows. In Section 2 we prove that, under the assumptions of Theorem \ref{teo2} or \ref{teo3}, $\Omega$ is an annular domain. In Section 3 we show that the stream function solves a semilinear elliptic problem under overdetermined boundary conditions. We also complete the proof of Theorems \ref{teo} and \ref{teo2} by making use of \cite{brock0, brock}. The proof of Theorem \ref{teo3} needs a study of elliptic operators with singular coefficients, which is performed in Section 4. In Section 5 we use this study to apply the moving plane procedure and conclude the proof of Theorem \ref{teo3}. Some final comments and remarks are gathered in Section 6.

\medskip

{\bf Notation:} We define $d: \Omega \to \R^+$, 
\begin{equation} \label{d} d(x)= d(x, \partial \Omega) = \min \{ |x-p|, \ p \in \partial \Omega \}. \end{equation}
 
At a regular point of $\partial \Omega$ we denote $\nu$ the exterior normal unit vector and $\tau$ the tangent unit vector (counterclockwise, for instance).
 
Given a vector $x \in \R^N$, we denote by $x_i$ its $i$-th component. If  $x=(x_1, x_2)\in \R^2$ we use the standard notation $x^{\perp} =(-x_2,x_1)$. 

\section{The set $\Omega$ is an annular domain}

In this section we prove that, under the assumptions of Theorem \ref{teo2} or Theorem \ref{teo3},  the domain $\Omega$ is an annular domain. To start with, let us write:
\begin{equation} \label{domains} \Omega= G_0 \setminus \cup_{i=1}^n \overline{G_i},\end{equation}
where $G_i$ are $C^1$ bounded and simply connected domains, $\overline{G_i} \subset G_0$, $\overline{G_i} \cap \overline{G_j}= \emptyset$ if $i\neq j$, $i\geq 1$, $j\geq 1$. If we denote by $\Gamma_i$ the boundaries of $G_i$, we have that

$$ \partial \Omega = \cup_{i=0}^n \Gamma_i.$$

Observe that the conditions $div \, u=0$ and $u=0$ on $\Gamma_i$ (actually $u \cdot \nu =0$ on $\Gamma_i$ would be enough) readily imply the existence of a stream function, $u = \nabla^{\perp} \phi$. We point out that, under the assumptions of Theorem \ref{teo2}, $\phi$ is defined in the whole euclidean plane. Moreover, as commented in the introduction, $\phi$ satisfies:

\begin{equation} \label{parallel} \nabla (\Delta \phi) \parallel \nabla \phi =0 \mbox{ in } \Omega. \end{equation}

The assumption (B2) of Theorem \ref{teo2} implies that:

\begin{equation} \label{key} \left \{ \begin{array}{l} D_j u=0 \mbox{ in } \Omega \mbox{ for } j <k, \\ \displaystyle \frac{\partial^k u}{\partial \tau^i \partial \nu ^{k-i}}=0 \mbox{ if } 1 \leq i \leq k, \quad  \frac{\partial^k u}{\partial \nu ^k}  \neq 0 \ \mbox{ in } \partial \Omega. \end{array} \right. \end{equation}

Here $D_j$ denotes any derivative of order $j$. As a consequence,

\begin{equation} \label{key2} \left \{ \begin{array}{l} D_j \phi=0 \mbox{ in } \Omega \mbox{ for } j <k+1, \\ \displaystyle \frac{\partial^{k+1} \phi}{\partial \tau^i \partial \nu ^{k+1-i}}=0 \mbox{ if } 1 \leq i \leq k+1, \quad  \frac{\partial^{k+1} \phi}{\partial \nu ^{k+1}}  \neq 0 \ \mbox{ in } \partial \Omega. \end{array} \right. \end{equation}

Instead, under assumption (C2) of Theorem \ref{teo3} we have that:

\begin{equation} \label{key3} \left \{ \begin{array}{l} \nabla \phi=0 \mbox{ in } \Omega, \\ \displaystyle \frac{\partial^{2} \phi}{\partial \tau^i \partial \nu ^{2-i}}=0 \mbox{ if } i=1, \ 2, \quad  \frac{\partial^{2} \phi}{\partial \nu ^{2}}  \neq 0 \ \mbox{ in } \partial \Omega. \end{array} \right. \end{equation}

Let us define:
$$ \Omega_{\e}= \{ x \in \Omega: \ d(x) > \e \}.$$

Then there exists $\e_0>0$ such that, for $\e \in (0, \e_0)$ we have that $\partial \Omega_\e$ is $C^2$ and has $n+1$ connected components. Moreover, \eqref{key2} or \eqref{key3} implies that the normal derivative of $\phi$ does not vanish on $\partial \Omega_\e$. This, in turn, implies that the tangential component of $u$ does not vanish in $\partial \Omega_\e$:

\begin{equation} \label{key4} \partial \Omega_\e= \cup_{i=0}^n \Gamma_i(\e) \mbox{ and } u \cdot \tau \neq 0 \mbox{ in } \partial \Omega_\e. \end{equation}

We start with the following general result, that must be well known.

\begin{lemma} \label{degree} Let $\Gamma \subset \R^2$ a $C^1$ closed curve and $F: \Gamma \to \R^2$ a continuous vector field. Assume that $ F \cdot \tau  \neq 0 \ \mbox{ in } \Gamma.$ Then,
$$ i_K(F, \Gamma) = 1,$$
where $i_K$ stands for winding number or Poincar\'{e} index.

\end{lemma}

\begin{figure}[h]
	\centering 
	\begin{minipage}[c]{100mm}
		\centering
		\resizebox{100mm}{100mm}{\includegraphics{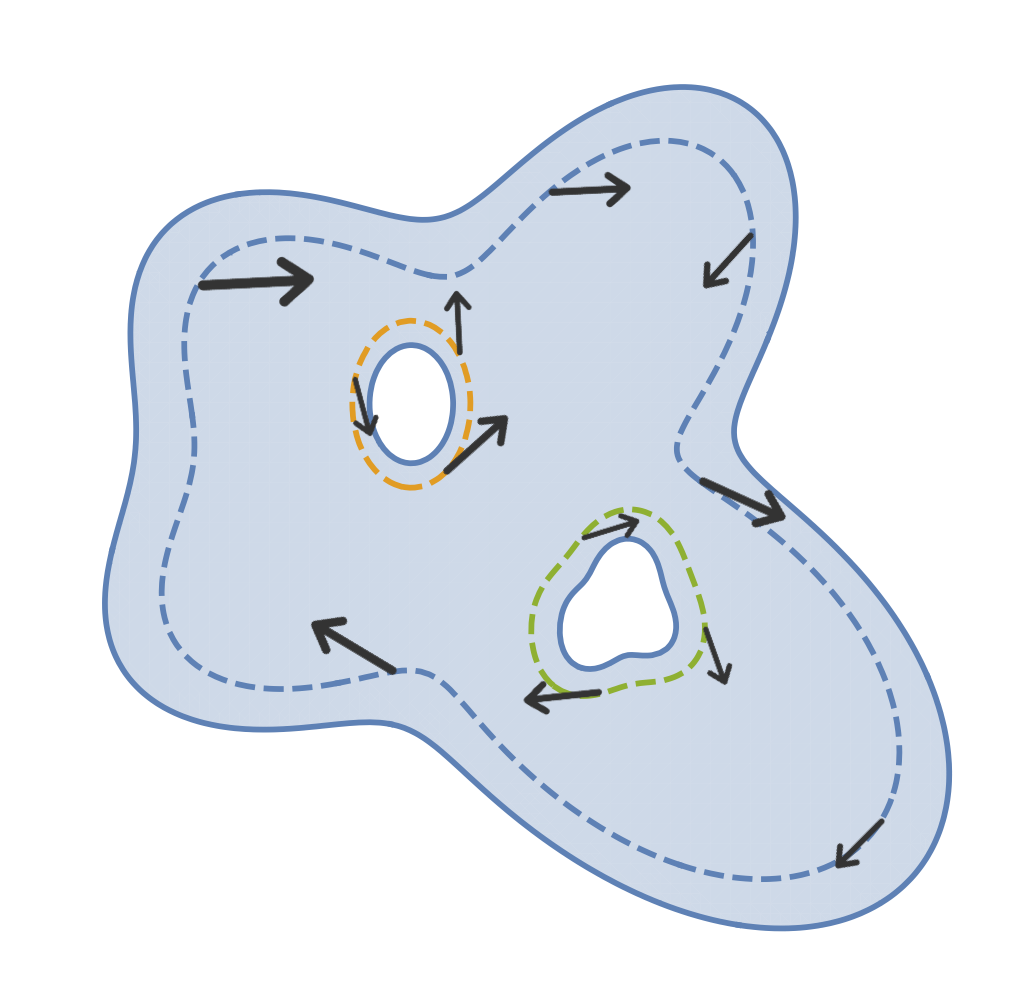}}
	\end{minipage}
	\caption{The domain $\Omega$ with, eventually, two holes, and $\Omega_\e$ (with dashed boundary). On $\partial \Omega_\e$ the tangential component of $u$ is always different from 0, so the winding number on each connected component is 1.}
\end{figure}

\begin{proof} 
	
By density we can assume that $\Gamma$ is $C^2$. Observe that $F(x) \cdot \tau(x)$ preserves its sign along $\Gamma$; let us assume that $F(x) \cdot \tau(x)>0$. We define the homotopy: $H:[0,1] \times \Gamma \to \R^2$, $H(\lambda, x)= \lambda F(x) + (1-\lambda)\tau(x)$. Clearly,

$$ H(\lambda, x) \cdot \tau(x)=  \lambda F(x) \cdot  \tau(x) + (1-\lambda)>0.$$
Then, by the homotopy invariance of $i_K$, it suffices to compute $ i_K(\tau, \Gamma)$.

Assume for simplicity that the length of $\Gamma$ is $ 2 \pi$, and define $\gamma:[0, 2 \pi] \to \Gamma$, $\gamma(0)= \gamma(2\pi)$ a parametrization with respect to the arc length, $\gamma'(s)= \tau(\gamma(s))$. Then we can compute the winding number:

$$ i_K(\tau, \Gamma) = \frac{1}{2\pi} \int_0^{2\pi} \tau(\gamma(s))^\perp \cdot \Big ( \frac{d}{ds} \tau(\gamma(s)) \Big ) \, ds = \frac{1}{2\pi} \int_0^{2\pi} \gamma'(s)^\perp \cdot \gamma''(s) \, ds $$$$= \frac{1}{2\pi} \int_{\Gamma} K(x) \, ds_x=1,$$

where $K(x)$ is the curvature of $\Gamma$ at the point $x$.

\end{proof}

With this lemma in hand we can determine the topology of $\Omega$ as follows.
	
\begin{proposition} \label{annular0} Under the assumptions of Theorem \ref{teo2} or Theorem \ref{teo3}, $\Omega$ is an annular domain, that is, \eqref{domains} holds with $n=1$.
\end{proposition}	

\begin{proof} Take $\e \in (0, \e_0)$ such that \eqref{key4} is satisfied. Recall that $u$ does not vanish in $\Omega$, and hence $$ d(u, \Omega_\e,0)=0.$$
Here $ d(u, \Omega_\e,0) $ stands for the Brouwer degree of $u$ in $\Omega_\e$.
Observe now that, by \eqref{key4} and Lemma \ref{degree},

$$ d( u, \Omega_\e,0) = i_K(u, \Gamma_0(\e)) - \sum_{i=1}^n i_K(u, \Gamma_i(\e)) = 1 - n.$$

Hence $n=1$ and the proof is complete.

\end{proof}

\begin{remark} \label{delicate} The same ideas apply directly if $u \cdot \tau \neq 0$ on $\partial \Omega$, as in \cite[Theorem 1.13]{hamel4}. In this case one does not need to make use of the perturbed domain $\Omega_\e$, and the computation of the degree of $u$ on $\Omega$, together with the lack of stagnation points and Lemma \ref{degree},  implies that $n=1$.
\end{remark}

\section{Proof of Theorems \ref{teo} and \ref{teo2}.}

In view of Proposition \ref{annular0}, Theorem \ref{teo2} reduces to Theorem \ref{teo}. In this section we prove that the stream function solves a semilinear elliptic PDE under overdetermined boundary conditions. This is also true in the setting of Theorem \ref{teo3}. Then, we use the results of \cite{brock0, brock} to complete the proof of Theorem \ref{teo}.

Observe that the stream function $\phi$ must be constant on each connected component $\Gamma_0$, $\Gamma_1$. By adding constants or multiplying $\phi$ by a constant number, we can assume that $\phi=0$ on $\Gamma_0$ and $\phi =1$ on $\Gamma_1$.

\begin{proposition} \label{prop} Under the assumptions of Theorem \ref{teo}  or \ref{teo3}, there exists $f \in C^1(0,1) \cap C[0,1]$ such that the function $\phi$ satisfies:
	
	\begin{equation} \label{od-bis} \left \{ \begin{array}{rll}  -\Delta \phi  &= f(\phi), \ \ \ \phi  \in (0,1)  &\mbox{ in } \Omega, \\ \phi  &=0,  \qquad \ \nabla \phi=0, & \mbox{ in } \Gamma_0, \\ \phi &= 1, \qquad \nabla \phi=0,  &\mbox{ in }\Gamma_1. \end{array} \right. \end{equation}	
Moreover:
\begin{enumerate}
	\item Under the assumptions of Theorem \ref{teo}, $\phi$ is defined in $\R^2$ and $\phi=1$ in $\overline{G_1}$, $\phi=0$ in $\R^2 \setminus G_0$. Moreover, $f(0)=f(1)=0$.
	\item Under the assumptions of Theorem \ref{teo3}, $f(0)<0<f(1)$.
\end{enumerate} 

\end{proposition}

\begin{proof}
The proof will be given in several steps.

\medskip

{\bf Step 1.} For any $c \in [0,1]$ the level sets $\phi_c= \{ x \in \overline{\Omega}: \ \phi(x) =c \}$ are continuous curves, which are $C^2$ if $c \in (0,1)$.

\medskip

Since $\phi$ has no critical points in $\Omega$ it is clear that $\phi_0 = \Gamma_0$, $\phi_1= \Gamma_1$. Given $c \in (0,1)$, the set $\phi_c$ is $C^2$ from the Implicit Function Theorem. As a consequence,
$$ \phi_c = \cup_{i=1}^m \gamma_i,$$
where $\gamma_i$ are closed $C^2$ curves in $\Omega$. 

We first show that each curve $\gamma_i$ contains $\overline{G_0}$ in its interior, denoted as $I(\gamma_i)$. Otherwise, the level set $\phi_c$ contains the boundary of $I(\gamma_i)$, and then $\phi$ would present a local maximum or minimum in $I(\gamma_i)$. Recall that this is not possible since $\nabla \phi \neq 0$ in $\Omega$.

We now prove that $m=1$, that is, $\phi_c$ is connected. Otherwise assume that $\gamma_1$ and $\gamma_2$ are closed curves in $\phi_c$. Without loss of generality, we can assume that $\gamma_2 \subset I(\gamma_1)$. Then, again, $\phi$ would have a local maximum or minimum in $I(\gamma_1) \setminus \overline{I(\gamma_2)}$, which is a contradiction. This concludes the proof of Step 1.

\medskip {\bf Step 2. } There exists a continuous function $f:[0,1]\to \R$ such that $-\Delta \phi = f(\phi)$. 

\medskip 

For any $c \in [0,1]$, take $x \in \phi_c$ and define $f(c)= -\Delta \phi(x)$. First, we briefly show that $f$ is well defined, that is, it does not depend on the choice of $x \in \phi_c$. 

In the setting of Theorem \ref{teo}, due to the lack of regularity of $\partial \Omega$, we need to treat separately the cases $c=0$ and $c=1$. Observe that for any $x \in \partial \Omega$, $\Delta \phi(x)=0$, so $f(0)=f(1)=0$ is well defined.

For $c \in (0,1)$  \eqref{parallel} implies that $\nabla \Delta \phi$ is parallel to $\nabla \phi$. Since $\phi_c$ is connected and $C^2$, we conclude that $\Delta \phi$ is constant on $\phi_c$. Observe that this holds also for $c=0$ and $c=1$ under the setting of Theorem \ref{teo3} since, by continuity, \eqref{parallel} is satisfied in $\overline{\Omega}$.

By construction, we have that $-\Delta \phi = f(\phi)$ in $\Omega$. Let us now show that $f$ is continuous. Let $c_n \in [0, 1]$, $c_n \to c$. Let $x_n \in \phi_{c_n}$; up to a subsequence we can assume that $x_n \to x \in \overline{\Omega}$. By continuity of $\phi$, we have that $\phi(x)=c$. Moreover, by the continuity of $\Delta \phi$, we have that:

$$ f(c_n) = \Delta \phi(x_n) \to \Delta \phi(x)= f(c).$$

\medskip {\bf Step 3. } The function $f$ is $C^1$ in $(0,1)$.

\medskip

This is contained in \cite{hamel1, hamel2, hamel3, hamel4}; let us sketch a proof here for the sake of completeness.  Let $c \in (0, 1)$ and $x \in \phi_c \subset \Omega$. For some $\delta >0$, define $ \sigma: (-\delta, \delta) \to \Omega$, 

$$ \left \{ \begin{array}{ll}  \sigma'(t)= \nabla \phi (\sigma(t)), & t \in (-\delta, \delta), \\ \sigma(0)=x. &  \end{array}  \right.$$

Since $\nabla \phi(x) \neq 0$, we have that $g'(0)>0$, where $g= \phi \circ \sigma$. By taking a smaller $\delta>0$ if necessary and a suitable $\e_i >0$, we have that $g: (-\delta, \delta) \to (c-\e_1, c + \e_2 )$ is a diffeomorphism. Now we obtain that
$$f|_{(c-\e_1, c+\e_2)}= - \Delta \phi \circ \sigma \circ g^{-1}.$$

Since $\phi$ is $C^3$ in $\Omega$, then $f|_{(c-\e_1, c+\e_2)}$ is $C^1$ function.

\medskip {\bf Step 4. } Conclusion.
 
\medskip

Under the assumptions of Theorem \ref{teo} , we already saw in Step 2 that $f(0)=f(1)=0$. Moreover, since $u=0$ on $\R^2 \setminus \Omega$, we have that $\phi=1$ in $\overline{G_1}$ and in $\phi=0$ in $\R^2 \setminus G_0$.

In the setting of Theorem \ref{teo3}, \eqref{key3} holds. As a consequence, for any $x \in \Gamma_0$,

$$f(0)=- \Delta \phi(x)= -\frac{\partial^2 \phi}{\partial \nu^2}(x) \neq 0.$$

Since $\phi >0 $ in $\Omega$ we conclude that $f(0) <0$. Analogously, we can show that $f(1)>0$.

\end{proof}

\subsection{Proof of Theorem \ref{teo}}
 Let $B(0,R)$ be an euclidean ball containing $\overline{\Omega}$ and let us consider the function $\phi$ defined in $B(0,R)$. Let us recall that $\phi=0$ in $B(0,R) \setminus \{ G_0\}$ and $\phi=1$ in $G_1$.  Obviously $\phi$ is a nonnegative $C^2$ solution of the problem:
\begin{equation} \label{BR} \left \{ \begin{array}{ll} - \Delta {\phi} = f({\phi}), & \mbox{ in } B(0,R), \\ {\phi}=0 & \mbox{ in } \partial B(0,R). \end{array} \right.\end{equation}
	
Moreover, $f$ is a continuous function and $\nabla \phi(x) \neq 0$ if $x \in \Omega$. We now apply the local symmetry results of F. Brock (\cite{brock}) to $\tilde{\phi}$. For convenience of the reader, we state here the result that we need from \cite{brock}, in a version which is suited for our purposes.

\begin{theorem}[\cite{brock}] \label{brock}Let ${\phi}$ be a nonnegative weak solution of the problem \eqref{BR}, where $f$ is a continuous function. Assume also that the set $D= \{x \in B(0,R): \ 0< {\phi}(x) < \sup {\phi} \}$ is an open set and that ${\phi}$ is $C^1$ in $D$. Then ${\phi}$ is locally symmetric in any direction, namely:	
$$ D = A \cup S, \ \  A= \bigcup_{k \in K} A_k,$$
where:

\begin{enumerate} \item $K$ is a countable set, 
	\item $A_k$ are disjoint open annuli or balls and ${\phi}$ is radially symmetric and decreasing in each domain $A_k$, 
	\item $\nabla {\phi} =0$ on the set $S$.
	\end{enumerate}
\end{theorem}

Theorem \ref{brock} follows from \cite[Corollary 7.6 and Theorem 7.2]{brock}. Following the notation of \cite{brock}, in our case $G(x,z)= |z|^2$ and $f_2=f_3=0$. See also \cite[Theorem 6.1]{brock} for a characterization of local symmetry in any direction.

\medskip

The proof of Theorem \ref{teo} follows at once from Theorem \ref{brock}. Indeed, in our case $D= \Omega$. Recall moreover that $\nabla \phi \neq 0$ in $\Omega$, hence $S= \emptyset$. Since $\Omega$ is connected we conclude that the union has only one term, that is, $\Omega$ is an annulus, and $\phi$ is radially symmetric and decreasing.

\section{Some aspects of linear elliptic operators with singular coefficients}

Let us recall that in the setting of Theorem \ref{teo3}, $f(0)<0<f(1)$. Therefore, the argument of the previous section does not work anymore: if we extend the function $\phi$ constantly outside $\Omega$, we do not get a solution of \eqref{BR}. 

We are able to conclude in this case by using a moving plane technique. As commented in the introduction, the main difficulty is that $f$ need not be Lipschitz continuous at the boundary values. As a consequence, one needs to deal with elliptic operators with coefficients that are singular at $\partial \Omega$. The study of such operators is the scope of this section.

More precisely, we fix a bounded domain $\Omega \subset \R^N $ ($N\geq 2$) with $C^2$ boundary, and we consider operators of the form:
$$ L= - \Delta  + c(x).$$
Maximum principles are known to hold for the operator $L$ if $c \in L^p(\Omega)$ with $p >N/2$. However, we are interested in coefficients satysfying:
$$ c(x) d(x) \in L^\infty(\Omega).$$
We have not been able to find an explicit reference on maximum principles for such operators. In what follows 
we will be concerned with the operator $L$ acting on functions defined in subdomains $\Omega_0 \subset \Omega$ with Lipschitz boundary. 

\subsection{Principal eigenvalue and maximum principle}
Let us define the associated quadratic form:
$$ Q: H_0^1(\Omega) \to \R, \ \ Q(\psi)= \int_{\Omega} |\nabla \psi|^2 + c(x) \psi(x)^2.$$

The above expression is well defined thanks to the Hardy inequality, that we recall here:

\medskip {\bf Hardy inequality: }{\it For any bounded domain $\Omega \subset \R^N$ with $C^2$ boundary, there exists a constant $C>0$ such that: }
$$ \int_{\Omega} \frac{| \psi(x) |^2}{d(x)^2} \, dx \leq C \int_{\Omega} |\nabla \psi(x)|^2 \, dx  \quad \forall \ \psi \in H_0^{1}(\Omega).$$ 
{\it In particular, the linear map $T$: }
$$ T: H_0^{1}(\Omega) \to L^2(\Omega), \ \ T(\psi) = \frac{\psi}{d},$$ 
{\it is well defined and continuous.}

The Hardy inequality holds also for Lipschitz domains $\Omega$ (see \cite{Necas}), but this will not be used in this paper.

\medskip 
For a Lipschitz domain $\Omega_0$ we say that $\omega \in H^1(\Omega_0)$ solves $L(\omega) \geq 0$ in a weak sense if for any $\psi \in H_0^1(\Omega_0)$, $\psi \geq 0$,
$$ \int_{\Omega_0} \nabla \omega \cdot \nabla \psi + c(x) \omega \psi \geq 0.$$

The above expression is well defined thanks to the Hardy inequality. Let us point out moreover that we can define the trace of $\psi \in H^1(\Omega_0)$ as a function in $L^2(\partial \Omega_0)$, see for instance \cite[Theorem 3.37]{mclean}. 

\medskip The next proposition deals with the first eigenvalue and the weak maximum principle for this kind of operators, in the spirit of \cite{BNV}

\begin{proposition} \label{propmp} The following assertions hold true:

\begin{enumerate}
	\item For any Lipschitz domain $\Omega_0 \subset \Omega$, the eigenvalue $$ \lambda_1(\Omega_0)= \inf \{ Q(\psi): \ \psi \in H_0^1(\Omega_0), \ \int_{\Omega_0} |\psi|^2 =1\}$$ is achieved.
	\item The corresponding eigenfunction $\phi_1$ is strictly positive (or negative) and is $C^{1,\alpha}$ locally in $\Omega_0$, for any $\alpha <1$. In particular the first eigenvalue is simple.

	\item There exists $r>0$ such that for any $q \in \partial \Omega$ and any Lipschitz domain $\Omega_0 \subset B(q,r) \cap \Omega$ we have that $\lambda_1(\Omega_0)>0$.
	\item If $\lambda_1(\Omega_0) > 0$ and $\omega \in H^1(\Omega_0)$ satisfies that $\omega \geq 0$ in $\partial \Omega_0$ and  $L(\omega) \geq 0$ in a weak sense, then $\omega \geq 0$. If moreover $\omega$ is a $C^1$ function in $\Omega_0$, then either $\omega > 0$ in $\Omega_0$ or $\omega =0$.
\end{enumerate}
	
\end{proposition}

\begin{proof} Assertion (1) follows from easily from the fact that if $\psi_n \weakto \psi$, then $\psi_n \to \psi $ and $c(x) \psi_n \weakto c(x) \psi $ in $L^2(\Omega)$. Then,
$$ \liminf_{n \to +\infty } \int_{\Omega} |\nabla \psi_n|^2 \geq \int_{\Omega} |\nabla \psi|^2, $$
$$ \int_{\Omega} \psi_n^2 \to \int_{\Omega} \psi^2, \ \ \int_{\Omega} c(x) \psi_n^2 = \int_{\Omega} (c(x) \psi_n) \psi_n  \to \int_{\Omega} c(x) \psi^2.$$

In order to prove (2), observe that $\phi_1$ is a weak solution of the problem:

\begin{equation} \label{eigen} - \Delta \phi_1= -c(x) \phi_1(x) + \lambda_1 \phi_1(x).\end{equation}

The $C^{1,\alpha}_{loc}$ regularity of $\phi_1$ follows from local regularity estimates (take into account that $c(x) \in L^{\infty}_{loc}(\Omega)$).

We now prove that $\phi_1$ is positive. Observe that $|\phi_1|$ is also a minimizer for $Q$, and hence it is also a weak solution of the problem \eqref{eigen}. If $\phi_1$ changes sign, the function $(\phi_1)^+ = \frac{1}{2}(| \phi_1| + \phi_1) $ is a nontrivial solution.

Hence, it suffices to deal with the case of nonnegative solutions that are equal to $0$ at some point in $\Omega$. However this is impossible by the classical maximum principle since $c(x)$ belongs to $L^\infty$ in the interior of $\Omega$.

The positivity of the eigenfunction implies also that the eigenvalue is simple. Indeed, given $\phi_1$ and $\phi_2$ two different eigenfunctions, let us take a nontrivial function $\phi = \mu_1 \phi_1 + \mu_2 \phi_2$, where $\mu_i \in \R$ are such that $ \int_{\Omega} \phi=0$. But $\phi$ cannot change sign, which implies that $\phi=0$, that is, $\phi_1$ and $\phi_2$ are proportional.

%\medskip The proof of the inequality of (3) follows simply from the fact that $H_0^1(\Omega_0) \subset H_0^1(\Omega)$ up to extension by $0$. Moreover, if $\lambda(\Omega_0)=  \lambda(\Omega_1)$, then the extension of the eigenfunction $\tilde{\phi}_1$ associated to $\Omega_0$ serves as a minimizer for $\lambda_1(\Omega)$. By (2), such extension cannot vanish inside $\Omega$, which is possible only if $\Omega_0 = \Omega$.

\medskip We now prove (3). By making use of the Hardy inequality, if $\psi \in H_0^1(\Omega_0)$,

$$ \int_{\Omega_0} |c(x) \psi(x) |^2  \, dx \leq \| c(\cdot) d(\cdot)\|_{L^{\infty}} \int_{\Omega_0} \frac{| \psi(x) |^2}{d(x)} \, dx \leq \| c(\cdot) d(\cdot)\|_{L^{\infty}} r \int_{\Omega} \frac{| \psi(x) |^2}{d(x)^2} \, dx$$$$ \leq \| c(\cdot) d(\cdot)\|_{L^{\infty}} C \, r \int_{\Omega} |\nabla \psi(x)|^2 \, dx.$$

It suffices to take $r$ such that $\| c(\cdot) d(\cdot)\|_{L^{\infty}} C \,  r<1$ to conclude.

\medskip We now prove (4). Take $\omega^- = \max \{- \omega, 0 \}$ which belongs to $H_0^1(\Omega)$ by the boundary assumption on $\omega$. We test the weak inequality $L(\omega) \geq 0$ against $\omega^-$, and obtain that:
$$ Q(\omega^-) \leq 0.$$

Since $\lambda_1(\Omega_0) >0$, we conclude that $\omega^- =0$. If now $\omega$ is of class $C^1$ and is not $0$, then $\omega>0$ by the usual maximum principle (recall, again, that inside $\Omega$ the coefficient $c(x)$ is in $L^{\infty}$). This concludes the proof.

\end{proof}

\subsection{Hopf lemma and Serrin corner lemma}

In order to prove the Hopf lemma and the Serrin corner lemma for singular elliptic operators, we will use the following comparison functions.

\begin{lemma} \label{comparison functions}
	Given $N \in \N$, $C>0$ and $r>0$, there exists $r_0>0$ and two nonnegative $C^\infty$ solutions $\psi_1$ and $\psi_2$ of the problems:
	
	\begin{equation} \label{hopf} \left \{ \begin{array}{rll}  - \Delta \psi_1(x)  + \displaystyle \frac{C}{r-|x|} \psi_1(x) & \leq 0 \   & \mbox{ if } x \in A(0; r_0, r), \\ \\ \psi_1(x)  =0, \ \ \partial_\nu \psi_1(x) & =-1.  &  |x| = r.  \end{array} \right. \end{equation}
	
\begin{equation} \label{corner} \left \{ \begin{array}{rll}  - \Delta \psi_2(x)  + \displaystyle \frac{C}{r-|x|} \psi_2(x) & \leq 0 \   & \mbox{ if } x \in A^+(0; r_0, r), \\ \\ \psi(x)  &=0  & \mbox{ if } |x| = r \mbox{ or } x_1=0  \end{array} \right. \end{equation}
Here $A(0; r_0, r)= \{ x \in \R^N, |x| \in (r_0,r)\}$ and $A^+(0; r_0, r)= A(0; r_0, r)\cap \{x_1>0\}$. Moreover, $\nabla \psi_2(p)=0$ but $$\frac{\partial^2 \psi_2}{\partial \eta^2} (p) > 0 \ \mbox{ for all } p \mbox{ with }\ p_1=0 \mbox{ and }|p|=r.$$ Here $\eta$ is any vector entering the domain non-tangentially, i.e., $\eta \cdot p<0$ and $\eta_1>0$. 
	
\end{lemma}

\begin{figure}[h]
	\centering 
	\begin{minipage}[c]{100mm}
		\centering
		\resizebox{80mm}{80mm}{\includegraphics{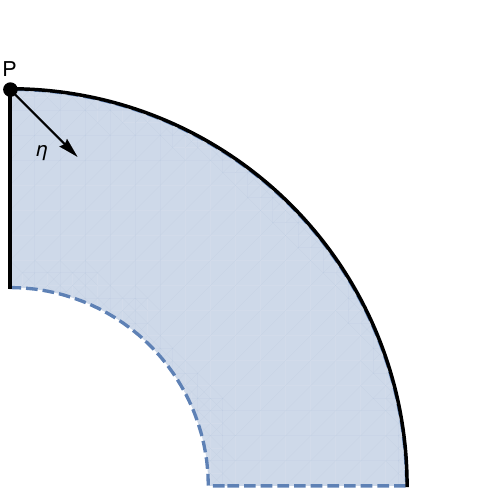}}
	\end{minipage}
	\caption{The function $\psi_2$ vanishes in the bold lines and is positive in the rest of the domain. At $p$ its gradient vanishes but its second derivative along $\eta$ is positive.}
\end{figure}

\begin{proof}
	
We first prove the following:

\medskip 

{\bf Claim: } Given $m \in \N$, $C>0$, $r>0$, there exists $r_0 \in (0, r)$ and $\rho \in C^{\infty}([r_0, r])$ such that:

\begin{equation} \label{hopf1} \left \{ \begin{array}{rll}  - \rho''(t) - (m-1) \displaystyle \frac{\rho'(t)}{t} + \displaystyle \frac{C}{r-t} \rho(t)  & \leq 0 \   & \mbox{ if } t \in [r_0, r),  \\  \\ \rho(t) &> 0 & \mbox{ if } t \in [r_0,r), \\ \\ \rho(r)  =0, \ \ \rho'(r) & =-1.  &    \end{array} \right. \end{equation}

We just give a explicit definition of $\rho: \R^+ \to \R$:
$$ \rho(t)= (r-t) + \Big( \frac{m-1}{t} + 2 C \Big) \frac{(r-t)^2}{2}.$$

Observe that $\rho(t) >0$ if $t <r$, $\rho(r)=0$, $\rho'(r)=-1$ and $\rho''(r)= \frac{m-1}{r} +2C$. Define:

$$ c(t)= \Big(\rho''(t) + (m-1) \frac{\rho'(t)}{t}  \Big) \frac{(r-t)}{\rho(t)}.$$

By definition, we have that:

\begin{equation} \label{ec} -\rho''(t) - (m-1) \frac{\rho'(t)}{t} + c(t) \frac{\rho(t)}{r-t}=0. \end{equation}

Observe also that

$$ \lim_{t \to r} \frac{\rho(t)}{r-t} =  - \rho'(1)=1.$$

As a consequence, taking limits in \eqref{ec}, we conclude that

$$ \lim_{t \to r} c(t)= 2C.$$

It suffices to take $r_0 \in (0, r)$ such that $c(t) > C $ for all $t \in (r_0,r)$ to conclude the proof of the claim.

\bigskip

The function $\psi_1$ can be easily defined as $\psi_1(x)= \rho(|x|)$, where $\rho$ is the function given in the claim with $m=N$. Clearly, $\psi_1$ satisfies the thesis of the lemma.

We define the function $\psi_2$ as:

$$ \psi_2(x)= x_1 \tilde{\rho}(|x|),$$
where now $\tilde{\rho}$ is the function given by the claim with $m = N+2$. Observe that:

$$ \Delta \psi_2= x_1 \Big ( \tilde{\rho}''(|x|) + (N-1)\frac{\tilde{\rho}'(|x|)}{|x|} \Big ) + 2 \tilde{\rho}'(|x|)\frac{x_1}{|x|} $$$$= x_1 \Big ( \tilde{\rho}''(|x|) + (m-1)\frac{\tilde{\rho}'(|x|)}{|x|} \Big ) \geq \frac{C}{r-|x|} \psi_2(x).$$
Fix a point $p$ such that $p_1=0$, $|p|=r$. Since $\psi_2$ vanishes for $x_1=0$ or $|x|=r$, $\nabla \psi_2(p)=0$. We now compute the second derivatives of $\psi_2$ at $p$.

It is clear, from direct computation, that if $j \neq 1, \ k \neq 1$,

$$ \frac{\partial^2 \psi_2}{\partial x_j \partial x_k } = x_1 \frac{\partial^2 \big ( \tilde{\rho}(|x|) \big ) }{\partial x_j \partial x_k } \Rightarrow \  \frac{\partial^2 \psi_2}{\partial x_j \partial x_k }(p) =0.$$
Moreover,
$$ \partial_{x_1} \psi_2(x)= \tilde{\rho}(|x|) + x_1^2 \frac{\tilde{\rho}'(|x|)}{|x|}.$$
From this we have that 
$$ \frac{\partial^2 \psi_2}{\partial x_1^2 }=  \rho'(|x|) \frac{x_1}{|x|} +   \partial_{x_1} \Big ( x_1^2 \frac{\tilde{\rho}'(|x|)}{|x|}\Big ) \Rightarrow \frac{\partial^2 \psi_2}{\partial x_1^2 }(p) =0. $$

Moreover, if $j \neq 1$,

$$ \frac{\partial^2 \psi_2}{\partial x_j \partial x_1 } = \rho'(|x|) \frac{x_j}{|x|} +   x_1^2 \partial_{x_j} \Big ( \frac{\tilde{\rho}'(|x|)}{|x|}\Big ) \Rightarrow $$ 
$$ \frac{\partial^2 \psi_2}{\partial x_j \partial x_1 } (p)= - \frac{p_j}{r} \neq 0.$$

Then we conclude by computing:

$$ \frac{\partial \psi_2}{\partial \eta^2}(p) = \sum_{j=2}^N \frac{\partial \psi_2}{\partial x_1 \partial x_j}(p) \eta_1 \eta_j =  \sum_{j=2}^N - \frac{p_j}{r} \eta_1 \eta_j = -\frac{1}{r} (p \cdot \eta) \eta_1>0.$$

\end{proof}

We now state the version of the Hopf lemma that is suited to our purposes.

\begin{proposition}[Hopf lemma for singular operators] \label{phopf} Let $\Omega \subset \R^N$ be a $C^2$ domain, and $c(x)$ satisfying that $c(x) d(x) \in L^{\infty}(\Omega)$. Let $B_r \subset \Omega$ be a ball of radius $r>0$, and $\omega \in C^1(\overline{B_r})$ solving
$$ - \Delta \omega + c(x) \omega \geq 0$$
in a weak sense. Assume that:
\begin{enumerate}
	\item $\omega \geq 0$ in $B_r$;
	\item $\omega(p)=0$ for some $p \in \partial B_r$.
\end{enumerate}

Then:
$$ \mbox{either } \partial_{\nu} \omega(p) <0 \ \mbox{ or } \omega =0 \mbox{ in } B_r$$

\end{proposition}

\begin{proof} If $p \notin \partial \Omega$, the conclusion follows from the usual Hopf lemma, since $c(x)$ belongs to $L^{\infty}$ in the interior of $\Omega$. We focus on the case $p \in \partial \Omega$. Observe that we can assume, by taking a smaller ball if necessary, that the radius $r$ is such that $\lambda_1(B_r)>0$ according to Proposition \ref{propmp}, (3). If $\omega$ is not identically equal to $0$, by Proposition \ref{propmp}, (4) we have that $\omega >0$ in $B_r$.  Let us assume for simplicity that the center of $B_r$ is the origin. Take now $\psi_1$ as in Lemma \ref{comparison functions} with $C= \| c(\cdot ) d(\cdot) \|_{L^{\infty}}$. Since $\omega >0$ in $B_r$, we can take $\e>0$ small enough such that $\e \psi_1 (x) < \omega(x)$ if $|x|=r_0$. 
	
Our intention is to compare $\e \psi_1$ with $\omega$. Let us point out that $d(x) \geq r-|x|$ in $B_r$. Moreover, by \eqref{hopf},
$$ L(\psi_1) \leq 0,$$
in $A(0; r_0, r)$. Then we conclude that that $\omega - \e \psi_1$ satisfies:

$$ L(\omega - \e \psi_1) \geq 0 \mbox{ in } A(0; r_0, r),$$

with $\omega - \e \psi_1 \geq 0$ in $\partial A(0;r_0,r)$. 

By Proposition \ref{propmp}, (4), we have that $\omega \geq \e \psi_1$. As a consequence, $ \partial_{\nu} \omega(p) \leq -\e.$

\end{proof}

Finally we state and prove a version of the Serrin corner lemma (\cite[Section 4, Lemma 2]{serrin}), adapted to our setting.

\begin{proposition}[Serrin corner lemma for singular operators] \label{pserrin} Let $\Omega \subset \R^N$ be a $C^2$ domain, and $c(x)$ satisfying that $c(x) d(x) \in L^{\infty}(\Omega)$. Let $B_r \subset \Omega$ be a ball of radius $r>0$, and $B_r^+$ a half ball. We can assume, without loss of generality, that 
$$B_r^+ = \{ x \in \R^N: \ |x|<r, \ x_1>0\}.$$
Let	$\omega \in C^2(\overline{B_r^+})$ be a weak solution of the inequality:
	$$ - \Delta \omega + c(x) \omega \geq 0$$
	 Assume that:
	\begin{enumerate}
		\item $\omega \geq 0$ in $B_r^+$;
		\item $\omega(p)=0$ for some $p \in  \partial B_r \cap \{x_1=0\}$;
		\item $\nabla \omega(p)=0$.
	\end{enumerate}
	
	Then:
	$$ \mbox{either } \frac{\partial^2 \omega}{\partial \eta^2} (p) > 0  \ \mbox{ or } \omega =0 \mbox{ in } B_r^+,$$
	
where $\eta \in \R^N$ is any unit vector with $\eta_1>0$, $p \cdot \eta <0$.
	
\end{proposition}

\begin{proof} If $p \notin \partial \Omega$, the conclusion follows from the usual Serrin lemma (see \cite[Section 4, Lemma 2]{serrin}), since $c(x)$ belongs to $L^{\infty}$ in the interior of $\Omega$.  We focus on the case $p \in \partial \Omega$. Observe that we can assume, by taking a smaller ball if necessary, that the radius $r$ is such that $\lambda_1(B^+_r)>0$ according to Proposition \ref{propmp}, (3). If $\omega$ is not identically equal to $0$, by Proposition \ref{propmp}, (4) we have that $\omega >0$ in $B^+_r$. 	
	
Take now $\psi_2$ as in Lemma \ref{comparison functions} with $C= \| c(\cdot ) d(\cdot) \|_{L^{\infty}}$.

At this point, recall that we are assuming $\omega >0$ in $B_r^+$, and observe that if $\omega(q)=0$ for some $q \in B_r \cap \{ x_1=0\}$, then $\partial_{x_1} \omega (q) >0$ by the classical Hopf lemma. As a consequence, there exists $\e>0$ such that $\e \psi_2 \leq \omega $ in $\partial A^+(0; r_0, r)$. 

We now compare $\e \psi_2$ with $\omega$. Let us point out that $d(x) \geq r-|x|$ in $A^+(0; r_0,r)$. Moreover, by \eqref{corner},
$$ L(\psi_1) \leq 0,$$
in $A^+(0; r_0, r)$. Then we conclude that $\omega - \e \psi_1$ satisfies:

$$ L(\omega - \e \psi_2) \geq 0 \mbox{ in } A^+(0; r_0, r),$$
with $\omega - \e \psi_2 \geq 0$ in $\partial A^+(0;r_0,r)$. 

By Proposition \ref{propmp}, (4), we have that $\omega \geq \e \psi_2$. Recall now that $\nabla \omega(p) = \nabla \psi_2(p)=0$. Define the function:
$ \kappa(t) =\omega(p + t \eta) - \e \psi_2(p + t \eta)$. Clearly, $\kappa(t)\geq 0$ for $t>0$, $\kappa(0)=0$ and $\kappa'(0)=0$. Hence $\kappa''(0) \geq 0$, which implies that:

$$ \frac{\partial^2 \omega}{\partial \eta^2} (p) \geq \ \e \, \frac{\partial^2 \psi_2}{\partial \eta^2} (p)>0.$$

\end{proof}

\section{Radial symmetry of $C^3$ solutions to overdetermined elliptic equations with non-Lipschitz nonlinearity. Conclusion of the proof of Theorem \ref{teo3}} 
In this section we prove the following result:
\begin{theorem} \label{reichel} Let $N \in \N$, $\Omega \subset \R^N$ a $C^2$ domain, 
	\begin{equation} \label{annular2} \Omega = G_0 \setminus \overline{G_1}, \end{equation}
	where $G_i$ are bounded simply connected domains and $\overline{G_1} \subset G_0$. Let $f \in C([0,1]) \cap C^1(0,1)$ and $\phi \in C^3(\overline{\Omega})$ be a solution of the overdetermined problem:
	
	\begin{equation} \label{od-2} \left \{ \begin{array}{rll}  -\Delta \phi  &= f(\phi), \ \ \phi  \in (0,1)  &\mbox{ in } \Omega, \\ \phi  &=0,  \quad \ \ \ \ \partial_\nu \phi=c_0, & \mbox{ in } \partial G_0, \\ \phi &= 1, \quad \ \ \partial_\nu \phi=c_1,  &\mbox{ in }\partial G_1. \end{array} \right. \end{equation}
If $c_i =0 $ we also assume that $f(i) \neq 0$, $i=0$, $1$. Then $\phi$ is a radially symmetric function with respect to a point $p \in \R^2$ and $G_i =B(p, R_i)$, $R_1 >R_0 >0$. Moreover, $\phi$ is strictly radially decreasing.
	
\end{theorem}

In view of Proposition \ref{prop}, Theorem \ref{teo3} follows inmediately from Theorem \ref{reichel}.

We point out that Theorem \ref{reichel} is a version of the result of \cite{reichel} (see also \cite{sirakov}) where $f$ is not necessary Lipschitz continuous around the values 0, 1. In exchange, we ask the solution to be $C^3$ up to the boundary, and also $f(i) \neq 0$ if $c_i=0$.

The next lemma is a crucial ingredient in the implementation of the moving plane method in this framework, together with the results of Section 4.

	\begin{lemma} \label{lemasingular} Under the assumptions of Theorem \ref{reichel}, there exists $C>0$ such that:
		
		$$ |f(\phi(x)) - f(\phi(y))| \leq \frac{C}{\min \{d(x), d(y)  \}} |\phi(x)- \phi(y)| \ \forall x, \ y \in \Omega.$$
	\end{lemma}
	
	\begin{proof} The proof is done in several steps.
		
		\medskip 
		
		{\bf Step 1: } There exists $c>0$ and a neighborhood of $\partial \Omega$ such that
		
		\begin{equation} \label{1bis} |\nabla \phi (x) | \geq c \, d(x). \end{equation}
	 
		\medskip If $c_i \neq 0$ for $i=0$ and/or $i=1$, the claim readily follows in a neighborhood of $\Gamma_i$. Instead, if $c_i=0$, we have that:
		
		$$c_i=0 \Rightarrow \nabla \phi(x) =0 \ \forall x \in \partial \Gamma_i   \Rightarrow \frac{\partial^2 \phi}{\partial \nu^2} \phi(x)= \Delta \phi(x)= -f(i) \neq 0 \ \forall x \in \Gamma_i.$$ With this we conclude the proof of step 1.
		
		\medskip
		
		{\bf Step 2: } There exists $C>$ such that
		\begin{equation} \label{estimate} |f'(\phi(x))| \leq \frac{C}{d(x)} \ \forall \, x \in \Omega. \end{equation}
		where $d(x)$ is defined in \eqref{d}.
		
		\medskip Let us fix $\e>0$ sufficiently small and denote $N_{\e}$ a tubular neighborhood of $\Omega$,
		
		$$ N_\e = \{p \in \Omega:\ d(p) \leq \e \}, \ \Omega_\e= \Omega \setminus N_\e.$$

		If $x \in \Omega_\e$ then $\phi(x) \geq \delta>0$ and hence \eqref{estimate} follows from the Lipschitz continuity of $f(t)$ for $t \geq \delta$. If instead $x$ belongs to $N_\e$, we can use estimate \eqref{1bis} and the $C^3$ regularity of $\phi$, to conclude:
		
		$$ |\nabla \Delta \phi(x)| = |f'(\phi(x))| |\nabla \phi(x)| \geq c \, d(x) |f'(\phi(x))|.$$
		
		From this \eqref{estimate} follows.
		
		\medskip 
		
		{\bf Step 3: } Conclusion. 
		
		\medskip 
		
		Take now $x$, $y$ two points in $\Omega$. If both points belong to $\Omega_\e$, then $\phi(x)$, $\phi(y)$ are both bigger than certain $\delta>0$. Since $f(t)$ is Lipschitz continuous for $t\geq \delta$, we are done.
		
		Let us assume now that at least one of the points belongs to $N_\e$, and take $r = \frac{1}{2} \min \{d(x), \ d(y) \} \leq \e/2$. Clearly, $ x, \ y \in \Omega_{r}$ and $\Omega_{r}$ is path connected if we have chosen $\e$ sufficiently small. Then, there exists a curve $$\gamma :[0,1] \to \Omega, \  \gamma(0)=x, \ \gamma(1)=y, \mbox{ and }  d(\gamma(t)) > r \ \forall \ t \in [0,1].$$
		
		We now use the mean value theorem:
		$$ f(\phi(x)) - f(\phi(y)) = f'(c)(\phi(x)- \phi(y)),$$
		for some $c$ between $\phi(x)$ and $\phi(y)$. By continuity, there exists $\xi \in [0,1]$ such that $\phi(\gamma(\xi))=c$. We now use \eqref{estimate} to conclude:
		
		$$ |f(\phi(x)) - f(\phi(y))| = |f'(\phi(\gamma(\xi))| \, |\phi(x)- \phi(y)| \leq \frac{C}{r} |\phi(x)- \phi(y)|. $$
		This finishes the proof of the lemma.
	\end{proof}

	%Observe that 
	%
	%$$ \nabla \Delta \phi = f'(\phi) \nabla \phi \Rightarrow |f'(\phi(x))|= \frac{|\nabla \Delta \phi(x)|}{|\nabla \phi(x)|}$$
	%
	%Hence, if $c_0 \neq 0$ for some point $x \in \partial \Omega$, it turns out that $f'(t)$ is bounded near $0$. Hence $f$ is Lipschitz continuous around $0$. Analogously, if $c_1 \neq 0$, $f(t)$ is Lipschitz continuous around $1$. Therefore, we will assume throughout the proof that 
	%
	%$$c_0=0 \Rightarrow \nabla \phi(x) =0 \ \forall x \in \partial \Omega.$$
	%
	%Since $\nabla \phi=0$ on $\partial \Omega$, we have: $$\partial_{\nu \nu} \phi(x)= \Delta \phi(x)= -f(0) \neq 0 \ \forall x \in \partial \Omega.$$ This implies in particular that $f(0)<0$. 
	
\begin{proof}[Proof of Theorem \ref{reichel}]
	
As commented previously, the proof follows from the moving plane argument.The argument can be summarized as follows: denoting by $\pi$ the reflection with respect to a hyperplane, Lemma \ref{lemasingular} implies that $w(x)= \phi(\pi(x)) - \phi(x)$ solves an equation $L(w) = 0$, where $L$ is an operator of the class studied in Section 4. In view of Propositions \ref{propmp}, \ref{phopf}, \ref{pserrin}, one can perform the moving plane argument in the spirit of  \cite{serrin}.

The proof follows the same argument as \cite[Theorem 3]{reichel}, so we will be sketchy. For the sake of clarity in the presentation, let us extend $\phi$ by $1$ in $G_1$. We fix one direction, say $x_1$, and define:

$$ \bar{\lambda} = \max \{x_1: \ x \in \overline{\Omega}\}.$$

For any $\lambda \in \R$, we define:
$$ H_\lambda = \{x_1= \lambda\}, \ H_\lambda^+ = \{x_1> \lambda\}, \ H_\lambda^- = \{x_1 < \lambda\}$$
$$ G_{i, \lambda}^\pm = G_i \cap H_{\lambda}^\pm, \ \Omega_{\lambda}^\pm = \Omega \cap H_{\lambda}^\pm,  $$$$ \Gamma_{i, \lambda}^\pm = \Gamma_i \cap  H_{\lambda}^\pm.$$
We also denote by $\pi_\lambda$ the reflection with respect to $H_\lambda$. Let us define:

$$I = \{\mu \in \R: \ \forall \lambda \geq \mu, \ \pi_{\lambda}(G_{i, \lambda}^+ \cup \Gamma_{i, \lambda}^+) \subset G_i,  \  (-1)^i \nu_1(p) > 0 \ \forall  p \in \Gamma_i \cap  H_{\lambda}, \ i =1,\ 2 \}.$$

Recall that $\nu_1(p)$ is the first component of the normal unitary vector exterior to $\Omega$. Observe that $I$ is bounded from below and that $\bar{\lambda} \in I$ trivially. Roughly speaking, the set $I$ represents the values of $\lambda$ for which the reflected caps of $G_i$ remain \emph{strictly insde} $G_i$. 

We denote by $ \mu_*$ the infimum of $I$, and clearly $\mu_* < \bar{\lambda}$. We also define the closed set:
$$J = \{\lambda \in I: \ \phi \circ \pi_{\lambda}(x) \geq \phi(x) \ \forall \ x \in G_{0,\lambda}^+  \}. $$

We claim that there exists $\e>0$ such that $(\bar{\lambda}- \e, \bar{\lambda}) \subset J$. This is the so-called \emph{initial step} in the proof of \cite[Theorem 3]{reichel} (see in particular Case 2). This is evident if $c_0 \neq 0$; if $c_0=0$, we have that:
$$\frac{\partial^2 \phi}{\partial \nu^2} \phi(x)= \Delta \phi(x)= -f(0) > 0.$$
This readily implies the claim.

\medskip 
Denote by $\lambda_*$ the minimum of $J$, that obviously satisfies $\lambda_* \geq \mu_*$. 

\medskip We now claim that $\lambda_*= \mu_*$. Indeed, the assumption $ \lambda_* > \mu_*$ gives a contradiction exactly as in \cite{reichel}. Observe that the contradiction is based on the maximum principle and the Hopf lemma on balls $B$ with $\overline{B} \subset \Omega$. In this case, $f(\phi)$ does not take the extremal values $0$ and $1$, and hence it is Lipschitz. Then one can use the standard procedure of the moving planes to obtain a contradiction. 

\medskip

Let us now focus on the extremal value $\lambda_*= \mu_*$. Since this value is fixed from now on, and for the sake of clarity, we drop the subscript $\lambda$ in the notation of what follows. Observe that $ w= \phi \circ \pi - \phi \geq 0$ in $G_0^+$, and moreover 
$$L (w) = 0 \ \mbox{ in } G_0^+ \setminus \pi(G_1^-),$$ where $ L= - \Delta + c(x)$, 
$$ c(x)= \left \{\begin{array}{ll} \displaystyle - \frac{ f(\phi(\pi(x))) - f(\phi(x))}{\phi(\pi(x)) - \phi(x)} & \mbox{ if } \phi(\pi(x)) \neq \phi(x), \\ 0 &  \mbox{ if } \phi(\pi(x)) = \phi(x). \end{array} \right. $$
By Lemma \ref{lemasingular}, 
\begin{equation} \label{casi} |c(x)| \leq \frac{C}{\min \{d(x), d(\pi(x))  \}}. \end{equation}

Since $\mu^*$ is the infimum of $I$ we have that $ \pi(G_{i}^+) \subset G_i$, $i = 1, \ 2$,  but at least one of the following alternatives is satisfied:
\begin{enumerate}
	\item \emph{Internal tangency.} There exists $p \in \Gamma_{i}^+$ such that $\pi(p)  \in \Gamma_i$, $i=1$ or $i=2$.
	\item \emph{Orthogonality of $\Gamma_i$ and $H$.} There exists $p \in \Gamma_{i}\cap H$ such that $\nu_1(p)=0$, $i=1$ or $i=2$.
\end{enumerate}

We treat each of these cases separately.

\medskip 
{\bf Case 1: Internal tangency.} Assume that we have internal tangency at a point $p \in \Gamma_{1}^+$, for instance. By the overdetermined boundary condition,
\begin{equation} \label{jotio} \partial_{\nu} w(p)=0. \end{equation} 
By $C^2$ regularity, we can take a ball $B_r$ of radius $r$ in $\Omega$ tangent to $\Omega$ at $p$. We can shrink that ball such that $B_r \in G_0^+ \setminus \pi(G_1^-)$. In other words, $B_r \subset \Omega \cap \pi(\Omega)$. From this, we have that:
\begin{equation} \label{casiultima} d(x) \geq r- |x|, \ d(\pi(x)) \geq r- |x| \ \mbox{ for any } x \in B_r. \end{equation}
From \eqref{casi}, we conclude that :
\begin{equation} \label{ultima}  |c(x)| \leq \frac{C}{r-|x|} \mbox{ in } B_r. \end{equation}
We now apply Proposition \ref{phopf} to the domain $B_r$ together with \eqref{jotio} to conclude that $w=0$ in $B_r$. The unique continuation principle implies that $w=0 $ in $G_0^+ \setminus \pi(G_1^-)$, which implies that $\Omega$ is symmetric with respect to $H$. Moreover, $\phi|_{H^+}$ is decreasing in $x_1$.

If the internal tangency occurs at a point $p \in \Gamma_{0}^+$, one can reason in an analogous manner.

\medskip

{\bf Case 2: Orthogonality of $\Gamma_i$ and $H$.} Let us now consider the case of a certain point $p \in \Gamma_{0} \cap H$ with $\nu_1(0)=0$. Reasoning as in \cite[pages 307-308]{serrin}, we conclude that the second order derivatives of $w$ at $p$ are zero:

\begin{equation} \label{jotio2} D^2w(p)=0. \end{equation} 

By $C^2$ regularity, we can take a ball $B_r$ of radius $r$ in $\Omega$ tangent to $\Omega$ at $p$. We can shrink that ball such that $B_r^+ \in G_0^+ \setminus \pi(G_1^-)$. In particular, $B_r \in \Omega \cap \pi(\Omega)$. As a consequence, also here \eqref{casiultima} is satisfied, and then the estimate \eqref{ultima} holds. We now apply Proposition \ref{pserrin} to the domain $B_r$ and $B_r^+$: taking into account \eqref{jotio2} we conclude that $w=0$ in $B_r^+$. As in the previous case, the unique continuation principle implies that $\Omega$ is symmetric with respect to $H$. Moreover, $\phi|_{H^+}$ is decreasing in $x_1$.

If the point $p$ belongs to $\Gamma_{1} \cap H$ one can reason in an analogous manner.

\bigskip

In both cases, we conclude that $\Omega$ is symmetric with respect to the $x_1$ direction. Since the direction $x_1$ is arbitrary and can be replaced by any other, we conclude that $\Omega$ is radially symmetric with respect to a point and $\phi$ is radially decreasing.
	
\end{proof}

\section{Further comments and remarks}

In this last section we just give some comments on the results presented in this paper. The first observation is that, under the assumptions of Theorem \ref{teo3}, the nonlinear term $f$ given by Proposition \ref{prop} need not be Lipschitz near the extremal values $0$, $1$. Let us give such an example. Take $\phi(x)= |x|^2 (2-|x|)^2$ defined in the annulus $A(0; 1,2) \subset \R^2$. It can be checked that $-\Delta \phi = f(\phi)$, with

$$ f:[0,1] \to \R, f(t)= 4  (4 \sqrt{s} + \sqrt{1 - \sqrt{s}}-3).$$

Observe that $f$ is not differentiable at $0$, $1$, but $u = (\nabla \phi)^{\perp}$ satisfies all conditions of Theorem \ref{teo3}.

\medskip In the setting of Theorems \ref{teo} or \ref{teo2} we can assure that $f$ is never Lipschitz continuous near $0$ or $1$. Recall that in such case we have $f(0)=0$, $f(1)=0$, and ${\phi}$ solves \eqref{BR}. But ${\phi}$ is constant in sets with non-empty interior, and this would be impossible if $f$ were Lipschitz continuous, by the maximum principle.

\medskip As commented in the introduction, Theorem \ref{teo} can be adapted to treat the case in which $\Omega$ is a punctured simply connected domain. Indeed, the following result follows. 

\begin{thmx}\label{teo4} Let $u: \R^2 \to \R^2$ be a compactly supported $C^1$ solution of \eqref{1}, and define $\Omega = \{x \in \R^2:\ u(x) \neq 0\}$. We assume that:	
	\begin{enumerate}
		
		\item[(D1)]  $\Omega= G_0 \setminus\{q\}$, where $G_0$ is a $C^2$ simply connected domain and $q  \in G_0$. 
		\item[(D2)] $u$ is of class $C^2$ in $\Omega$.
		
	\end{enumerate} 
	%Here $\nu$ stands for the outer normal unit vector in $\partial \Omega$. 
	%\
	
	\medskip Then there exists $R>0$ such that $\Omega=B(q,R) \setminus \{p\}$ and $u(x)$ is a circular vector field. Being more specific, there exist a certain function $V$ such that
	$$ u(x)= V(|x-q|) (x-q)^{\perp}.$$
\end{thmx}

For the proof, one can just follow the arguments of Section 3, replacing $\Gamma_1$ with $\{q\}$ in the notation. In the Step 1 of Proposition \ref{prop}, $\phi_c$ are closed and connected $C^1$ curves only for $c\in [0,1)$, and $\phi_1 = \{q\}$. Then we conclude the existence of a function $f \in C^1(0,1) \cap C[0,1]$ with:
	
	\begin{equation} \left \{ \begin{array}{rll}  -\Delta \phi  &= f(\phi), \ \ \ \phi  \in (0,1]  &\mbox{ in } G_0, \\ \phi  &=0,  \qquad \ \nabla \phi=0, & \mbox{ in } \Gamma_0. \end{array} \right. \end{equation}	
	
Moreover $\phi$ is defined in $\R^2$, $\phi=0$ in $\R^2 \setminus G_0$ and $f(0)=0$. This allows us to apply Theorem \ref{brock} to conclude.
	
\bigskip

In Section 5 we have proved a symmetry result for overdetermined elliptic problems defined in a annular domain. The main novelty with respect to \cite{reichel,sirakov} is that $f$ is not assumed to be Lipschitz continous on the boundary values. Of course the same ideas apply also in the framework of \cite{serrin}, and we obtain the following theorem.

\begin{theorem} \label{serrin} Let $\Omega \subset \R^N$ be a $C^2$ domain, $f \in C([0,+\infty)) \cap C^1(0,+\infty)$ and $\phi \in C^3(\overline{\Omega})$ be a solution of the overdetermined problem:
	
	\begin{equation} \label{od-1} \left \{ \begin{array}{rll}  -\Delta \phi & = f(\phi), \ \ \ \phi  >0  &\mbox{ in } \Omega, \\ \phi  & = 0, \ \ \ \ \partial_\nu \phi=c_0, & \mbox{ in } \partial \Omega. \end{array} \right. \end{equation}
	If $c =0 $ we also assume that $f(0) \neq 0$. Then $\Omega =B(p, R)$ and $\phi$ is a radially symmetric function with respect to $p$. Moreover, $\phi$ is strictly radially decreasing.
	
\end{theorem}

The proof follows exactly the same guidelines as in Theorem \ref{reichel}, with the obvious modifications.

%Moreover, a version of Theorem \ref{reichel} as in \cite{sirakov} could be stated. We did not  since our situation coming from the Euler flow fits perfectly in the setting of Theorem \ref{reichel}.

%\item As we have shown, the symmetry proof is rather different in the cases $k=1$ and $k\geq 2$. It is natural to ask then if one has different values of $k$ at each connected component of the boundary. Being more specific, assume that $\Omega = G_0 \setminus G_1$, and that assumption (A2) is replaced by:
%
%\begin{enumerate} \item[(A2')] There holds:
%	
%	\begin{itemize} \item $\frac{\partial^j u}{\partial \nu^j} = 0$ if $j < 2$ and $\frac{\partial^2 u}{\partial \nu^k}  \neq 0$ for all $x \in \partial G_0$.
%		\item $\frac{\partial u}{\partial \nu}  \neq 0$ for all $x \in \partial G_1$.
%		
%\end{itemize}
%\end{enumerate} 
%	
%By the above arguments, the stream function $\phi$ solves \eqref{od-bis} with $f(1)>0 = f(0)$. We have not found a way to deal with this situation. 

\bigskip

{\bf Data availability statement: } This manuscript has no associated data.

\end{document}